\documentclass[12pt,leqno,twoside]{amsart}

\usepackage{amssymb,amsmath,amsthm,soul,color}

\usepackage[noadjust]{cite}

\usepackage{t1enc}

\usepackage[utf8]{inputenc}
\usepackage{indentfirst}

\usepackage{todonotes}

\usepackage[left=2.5cm, right=2.5cm, top=2cm, bottom=2cm]{geometry}

\usepackage{url}
\usepackage{soul}

\usepackage{mathrsfs}

\usepackage{enumitem,graphicx}
\usepackage[colorlinks=true,urlcolor=blue,
citecolor=red,linkcolor=blue,linktocpage,pdfpagelabels,
bookmarksnumbered,bookmarksopen]{hyperref}
\usepackage[metapost]{mfpic}
\usepackage[normalem]{ulem}

\newtheorem{Th}{Theorem}[section]
\newtheorem{Prop}[Th]{Proposition}
\newtheorem{Lem}[Th]{Lemma}

\newenvironment{altproof}[1]
{\noindent
{\em Proof of {#1}}.}
{\nopagebreak\mbox{}\hfill $\Box$\par\addvspace{0.5cm}}

\newcommand{\wt}{\widetilde}

\newcommand{\vp}{\varphi}

\newcommand{\eps}{\varepsilon}

\newcommand{\R}{\mathbb{R}}

\newcommand{\cC}{{\mathcal C}}
\newcommand{\cD}{{\mathcal D}}

\newcommand{\cM}{{\mathcal M}}

\newcommand{\cS}{{\mathcal S}}

\newcommand{\Om}{\Omega}

\newcommand{\weakto}{\rightharpoonup}

\newcommand{\tu}{\widetilde{u}}
\newcommand{\tv}{\widetilde{v}}

\numberwithin{equation}{section}

\DeclareMathOperator*{\essinf}{ess\,inf}

\newcommand{\esssupp}{\mathrm{ess}\,\mathrm{supp}\,}

\begin{document}


\title{Normalized ground states of the nonlinear Schr\"{o}dinger equation with at least mass critical growth}

\author[B. Bieganowski]{Bartosz Bieganowski}
	\address[B. Bieganowski]{\newline\indent  	
	Faculty of Mathematics and Computer Science,		\newline\indent 	Nicolaus Copernicus University, \newline\indent ul. Chopina 12/18, 87-100 Toru\'n, Poland}	\email{\href{mailto:bartoszb@mat.umk.pl}{bartoszb@mat.umk.pl}}	

\author[J. Mederski]{Jaros\l aw Mederski}
	\address[J. Mederski]{\newline\indent
			Institute of Mathematics,
			\newline\indent 
			Polish Academy of Sciences,
			\newline\indent 
			ul. \'Sniadeckich 8, 00-656 Warsaw, Poland
			\newline\indent 
			and
			\newline\indent 
			Department of Mathematics,
			\newline\indent 
			Karlsruhe Institute of Technology (KIT), 
			\newline\indent 
			D-76128 Karlsruhe, Germany
	}
	\email{\href{mailto:jmederski@impan.pl}{jmederski@impan.pl}}
	
	\maketitle
	
	\pagestyle{myheadings} \markboth{\underline{B. Bieganowski, J. Mederski}}{
		\underline{Normalized ground states of the nonlinear Schr\"{o}dinger equation with at least mass critical growth}}

\maketitle

\begin{abstract} 
We propose a simple minimization method to show the existence of least energy solutions to the normalized problem
\begin{equation*}
\left\{ \begin{array}{l}
-\Delta u + \lambda u = g(u) \quad \mathrm{in} \ \R^N, \ N \geq 3, \\
u \in H^1(\R^N), \\
\int_{\R^N} |u|^2 \, dx = \rho > 0,
\end{array} \right.
\end{equation*}
where $\rho$ is prescribed and $(\lambda, u) \in \R \times H^1 (\R^N)$ is to be determined. The new approach based on the direct minimization of the energy functional on the linear combination of Nehari and Pohozaev constraints intersected with the closed ball in $L^2(\R^N)$ of radius $\rho$ is demonstrated, which allows to provide general growth assumptions imposed on $g$. We cover the most known physical examples and nonlinearities with growth considered in the literature so far as well as we admit the mass critical growth at $0$.
\medskip

   
\end{abstract}

\section{Introduction}

In this paper we are looking for solutions to the following nonlinear Schr\"odinger problem
\begin{equation}\label{eq}
\left\{ \begin{array}{l}
-\Delta u + \lambda u = g(u) \quad \mathrm{in} \ \R^N, \ N \geq 3, \\
u \in H^1(\R^N), \\
\int_{\R^N} |u|^2 \, dx = \rho > 0,
\end{array} \right.
\end{equation}
where $\rho$ is prescribed and $(u, \lambda) \in  H^1(\R^N) \times \R$ has to be determined. 

The following time-dependent, nonlinear Schr\"odinger equation
$$
\left\{
\begin{array}{l}
\mathrm{i} \frac{\partial \Psi}{\partial t} (t,x) = \Delta_x \Psi (t,x) + h(|\Psi(t,x)|)\Psi(t,x), \quad (t,x) \in \R \times \R^N, \\
\int_{\R^N} |\Psi(t,x)|^2 \, dx = \rho
\end{array}
\right.
$$
with prescribed mass $\sqrt{\rho}$ appears in nonlinear optics and the theory of Bose-Einstein condensates (see \cite{Akhmediev, Esry, Frantzeskakis, Malomed, Timmermans}). Solutions $u$ to \eqref{eq} correspond to standing waves $\Psi(t,x) = e^{-\mathrm{i}\lambda t} u(x)$ of the foregoing time-dependent equation. The prescribed mass represents the power supply in nonlinear optics or the number of particles in Bose-Einstein condensates.

Let us denote 
$$
\cS := \left\{ u \in H^1(\R^N) \ : \ \int_{\R^N} |u|^2 \, dx = \rho \right\},
$$
where $H^1(\R^N)$ is endowed with the usual norm $\|u\|=\big(|\nabla u|_2^2+|u|_2^2\big)^{1/2}$ and  $|\cdot|_q$  stands for the $L^q$-norm.
Under suitable assumptions provided below, solutions to \eqref{eq} are critical points of the energy functional $J : H^1(\R^N) \rightarrow \R$ given by
$$
J(u) := \frac12 \int_{\R^N} |\nabla u|^2 \, dx - \int_{\R^N} G(u) \, dx,
$$
where $G(u) := \int_0^u g(s) \, ds$, on the constraint $\cS$ with a Lagrange multiplier $\lambda \in \R$, i.e. they are critical points of the following functional
$$
H^1(\R^N) \ni u \mapsto J(u) + \frac{\lambda}{2} \int_{\R^N} |u|^2 \, dx \in \R
$$
with some $\lambda \in \R$. Recall that any critical point of the above functional lies in $W^{2,q}_{loc} (\R^N)$ for all $q < \infty$ and satisfies the following Pohozaev identity
$$
\int_{\R^N} |\nabla u|^2 \, dx = 2^* \int_{\R^N} G(u) - \frac{\lambda}{2} |u|^2 \, dx
$$
\cite{BerLions,MederskiBLproblem,Jeanjean,Shatah}.
On the other hand, all nontrivial critical points lie in the corresponding Nehari manifold, i.e.
$$
J'(u)(u)+\lambda \int_{\R^N} |u|^2 \, dx = 0
$$
and combining these two identities one can easily compute that any nontrivial solution satisfies
$$
M(u):= \int_{\R^N} |\nabla u|^2 \, dx - \frac{N}{2} \int_{\R^N} H(u) \, dx = 0,
$$
where $H(u):= g(u)u-2G(u)$, see e.g. \cite{Jeanjean}. Therefore we consider the following constraint
$$
\cM := \{ u \in H^1(\R^N) \setminus \{0\}:  M(u)=0\},
$$
which contains any nontrivial solution to \eqref{eq}. In our approach we also consider
$$
\cD := \left\{ u \in H^1(\R^N) \ : \ \int_{\R^N} |u|^2 \, dx \leq \rho \right\}
$$
and note that any nontrivial, (normalized) solution to $\eqref{eq}$ belongs to $\cS \cap \cM\subset\cD \cap \cM$. By a {\em normalized ground state} solution to  \eqref{eq} we mean a nontrivial solution minimizing $J$ among all nontrivial solutions. In particular, if $u$  solves \eqref{eq} and $J(u)=\inf_{\cS\cap\cM}J$, then $u$ is a normalized ground state solution.

Recall that, in the case of the pure power nonlinearity
\begin{equation}\label{Ex:powernon}
G(u)=\frac{1}{p}|u|^{p},
\end{equation}
the problem can be treated using variational methods available for the problem with fixed $\lambda > 0$ and by the scaling-type argument. This approach fails in the case of nonhomogeneous nonlinearities. 
In the $L^2$-subcritical case, i.e. where $G$ has growth $|u|^{p}$ with $2 < p < 2_* := 2+\frac{4}{N}$, one can use a minimization on the $L^2$-sphere $\cS$ in $H^1(\R^N)$ in order to obtain the existence of a global minimizer \cite{Stuart,Lions84}. In $L^2$-critical ($p = 2_*$) and $L^2$-supercritical and Sobolev-subcritical ($2_* < p < 2^*:=\frac{2N}{N-2}$) cases the minimization on the $L^2$-sphere does not work, if $p>2_*$ in \eqref{Ex:powernon}, then $\inf_\cS J=-\infty$, and this work is concerned with this problem. Our aim is to impose general growth condition on $g$ in the spirit of Berestycki and Lions \cite{BerLions,BerLionsII} and provide a new approach to study normalized ground state solution to \eqref{eq} and similar elliptic problems. 

We would like to mention that Jeanjean \cite{Jeanjean}, Bartsch and Soave \cite{BartschSoaveJFA, BartschSoaveJFACorr}, considered the problem \eqref{eq} with the nonlinear term satisfying the following Ambrosetti-Rabinowitz-type condition that there are $a>\frac4N$ and $b<2^*-2$ such that
\begin{equation}\label{eq:AR}
0 < aG(u)\leq H(u)\leq b G(u)\hbox{ for }u \in\R \setminus \{0\}.
\end{equation}
In \cite{Jeanjean} the solution has been found via the mountain pass argument, and  in \cite{BartschSoaveJFA, BartschSoaveJFACorr} the authors provided a mini-max approach in  $\cM$ based on the $\sigma$-homotopy stable family of compact subsets of $\cM$ and the minimax principle \cite{Ghoussoub}. The multiplicity of solutions to \eqref{eq} has been considered also in \cite{BartschdeVal} under the condition \eqref{eq:AR}. We would like to point out that the analysis of $L^2$-mass supercritical problems and recovering the compactness of Palais-Smale sequences is usually hard, since, for instance, the embedding of radial functions $H_{rad}^1(\R^N)\subset L^2(\R^N)$ is not compact and the argument is quite involved in $H_{rad}^1(\R^N)$, see e.g. \cite{Jeanjean,BartschSoaveJFA,BartschSoaveJFACorr,BartschJS2016}.
Another strategy to obtain the compactness is to show that the ground state energy map \eqref{gsm} is nonincreasing  respect to $\rho$ and strictly decreasing for some $\rho$, see e.g. \cite{BellazziniJeanjean, JeanjeanLuNorm}. 

 In our approach we do not work in $H_{rad}^1$,  the monotonicity of the ground state energy map \eqref{gsm} is not required to deal with the lack of compactness and we do not need to work with Palais-Smale sequences, so that we avoid the mini-max approach in $\cM$  involving a strong topological argument as in \cite{BartschSoaveJFA,BartschSoaveJFACorr,JeanjeanLuNorm,Ghoussoub}.
 
We work only with a minimizing sequence of $J$ on $\cD\cap\cM$ as we shall see later, and a wide class of nonlinearities is considered. Moreover, if $f$ is odd and $u\in \cD\cap\cM$, then the projection given by \eqref{def:r} of the Schwartz symmetrization $|u|^*$ of $|u|$ onto $\cM$ remains in $\cD\cap\cM$ and we do not encounter difficulties concerning the radial symmetry  appearing on the sphere $\cS$. In comparison to  a very recent and interesting work \cite{JeanjeanLuNorm} by Jeanjean and Lu,  we require that $H$ is of $\cC^1$-class, however our growth conditions are more general, in particular we assume a version of \eqref{eq:AR} with $a = \frac4N$ and $b=2^*-2$, which admits $L^{2_*}$-growth at $0$. Moreover, the strict monotonicity of \eqref{gsm} is just a simple consequence of our approach, see {\em Step 4} below. 

In order to state our assumptions, we recall the optimal constant $C_{N,p}>0$ in the Gagliardo-Nirenberg inequality
\begin{equation}\label{eqGN}
|u|_p \leq C_{N,p} |\nabla u|_2^\delta |u|_2^{1-\delta} \leq C_{N,p} \rho^{\frac{1-\delta}{2}} |\nabla u|_2^\delta,\quad\hbox{for }u\in H^1(\R^N),
\end{equation}
where $\delta = N \big( \frac{1}{2} - \frac{1}{p} \big)$.

Given functions $f_1,f_2:\R\to\R$. We introduce the following notation: 
 $f_1(s)\preceq f_2(s)$ for $s\in\R$ provided that 
  $f_1(s)\leq f_2(s)$ for all $s\in\R$ and for any $\gamma>0$ there is $|s|<\gamma$ such that $f_1(s)<f_2(s)$. An important property of the relation $\preceq$ is given in Lemma \ref{lem:ineq}.

Let us consider the following assumptions:
\begin{itemize}
	\item[(A0)] $g$ and $h:=H'$ are continuous and there is $c>0$ such that $$|h(u)|\leq c(|u|+|u|^{2^*-1})\hbox{  for }u \in\R.$$
	\item[(A1)] $\eta:=\limsup_{|u|\to 0} G(u)/|u|^{2+\frac{4}{N}}<\infty$.
	\item[(A2)] $\lim_{|u|\to\infty} G(u)/|u|^{2_*}=\infty$.
	\item[(A3)] $\lim_{|u|\to\infty} G(u)/|u|^{2^*}= 0$.
	\item[(A4)] $ \big(2+\frac{4}{N}\big) H(u)\leq  h(u)u$ for $u\in \R$. 
	\item[(A5)] $\frac4NG(u)\leq H(u)\leq (2^*-2)G(u)$ for $u \in\R$.
	\item[(A6)] $H(\zeta_0)>0$ for some $\zeta_0\neq 0$.
\end{itemize}

Note that (A0) implies that $J$ and $M$ are of class $\cC^1$. Moreover, assuming in addition  (A2) and (A5), $G(u)>0$ and $H(u)>0$ for $u\neq 0$, in particular (A6) holds. Indeed, in view of (A5) $\big(G(u)/u^{2_*}\big)'\geq 0$ and $\big(G(u)/u^{2^*}\big)'\leq 0$ for $u>0$, thus
\begin{eqnarray*}
	&&u^{2^*}G(1)\leq G(u)\leq  u^{2_*}G(1),\quad \hbox{ if }0<u<1,\\
	&&u^{2_*}G(1)\leq G(u)\leq  u^{2^*}G(1),\quad\hbox{ if }u\geq 1.
\end{eqnarray*}
Moreover,  (A2) and (A5) imply that $G(1)> 0$, hence  $G(u)>0$ and $H(u)>0$ for $u>0$. We argue similarly if $u<0$.
Then $\eta\geq 0$ if (A1) holds.
Now we show that
$\cM$ is a nonempty $\cC^1$-manifold, since $M'(u)\neq 0$ for $u\in \cM$, cf. \cite{Shatah}. Indeed, if $M'(u)=0$, then $u$ solves $-\Delta u = \frac{N}{4} h(u)$ and satisfies the Pohozaev identity $\int_{\R^N} |\nabla u|^2 \, dx = 2^* \frac{N}{4} \int_{\R^N} H(u) \, dx$. Since $u\in\cM$, we infer that $u=0$.

Observe that (A1) admits $L^{2}$-critical growth of $G$ close to $0$, however
(A2) excludes the pure $L^2$-critical case, e.g. \eqref{Ex:powernon} with $p=2_*$. Moreover (A3) excludes \eqref{Ex:powernon} with $p=2^*$. 
In our considerations we will also consider more restrictive inequalities $\preceq$ in (A4) and (A5)  denoted by (A4,$\preceq$) and (A5,$\preceq$) respectively.
Note that
(A5,$\preceq$) is a weaker variant of \eqref{eq:AR}.
(A5,$\preceq$) plays an important role in Lemma \ref{lem:SDineq}. 

Note that if (A1) holds with $\eta=0$ and (A2) is satisfied, and the inequality in (A4) is strict for $u\neq 0$, then $\frac4NG(u)<H(u)$ for $u\neq 0$ according to \cite[Lemma 2.3]{JeanjeanLuNorm} and we cover the growth conditions considered recently in \cite{JeanjeanLuNorm}. Finally, if (A5) holds with the first inequality replaced by $\preceq$, then (A6) is clearly satisfied. If $\eta=0$, then arguing similarly as in \cite[Lemma 2.3]{JeanjeanLuNorm}, we can show that (A0), (A2) and (A4,$\preceq$) imply that $\frac{4}{N}G(u)\preceq H(u)$ for $u\in\R$.

We recall the following definition of \textit{radial symmetry with respect to an affine subspace}, cf. \cite{Maris}. Fix an affine subspace $V$ of $\R^N$ and a function $u : \R^N \rightarrow \R$. Let $p_V : \R^N \rightarrow V$ denote the projection onto $V$. We say that $u$ is radially symmetric with respect to $V$ if there is $\tu : V \times [0,\infty) \rightarrow \R$ such that $u(x) = \tu (p_V(x), |x-p_V(x)|)$ for all $x \in \R^N$. If, in particular, $V=\{0\}$ then $u$ is radially symmetric.

The main result reads as follows.
\begin{Th}\label{th:main} Suppose that (A0)--(A5) are satisfied and 	
\begin{equation}\label{eq:Hstrict}
2^*\eta C_{N,2_*}^{2_*} \rho^{\frac{2}{N}} <1
\end{equation}
holds. Then there is $u\in \cD\cap\cM$ such that
	\begin{equation}
	J(u)=\inf_{\cD\cap\cM}J>0,
	\end{equation}
and if, in addition, $g$ is odd, then $u$ is radially symmetric. Suppose that (A5,$\preceq$) is satisfied.\\
$(a)$ If $g'(s)=o(1)$ as $s\to 0$, then $\inf_{\cD\cap\cM}J = \inf_{\cS \cap \cM} J$ and
$u \in \cS\cap\cM$ is a normalized ground state solution to \eqref{eq}. Moreover $u$ is radially symmetric with respect to some one-dimensional affine subspace $V$ in $\R^N$.\\
$(b)$ If $g$ is odd, then $\inf_{\cD\cap\cM}J = \inf_{\cS \cap \cM} J$ and $u\in \cS\cap\cM$ is a positive and radially symmetric normalized ground state solution to \eqref{eq}.  If $N\in\{3,4\}$, then it is sufficient to assume only that $H(s)\leq(2^*-2)G(s)$ for $s\in\R$ in (A5).
\end{Th}

In order to illustrate Theorem \ref{th:main} we provide the following examples and properties with regard to our assumptions.

\begin{itemize} 

	\item[(E1)] Suppose that $g$ satisfies (A0)--(A5) and $g$ is odd, e.g. \eqref{Ex:powernon} with $2_* < p < 2^*$. Then $g$ is of class $\cC^1$ on $(-\infty, 0) \cup (0, \infty)$ and note that $g'(\zeta) > 0$ for some $\zeta > 0$. Assume for simplicity that $\zeta = 1$. We define $\widetilde{g} : \R \rightarrow \R$ such that $\widetilde{g}(0)=0$ and
	$$
	\widetilde{g}'(s) = \left\{ \begin{array}{ll}
	g'(1) |s|^{2^*-2} & |s| \leq 1 \\
	g'(s) & |s| > 1,
	\end{array} \right.
	$$
	Then $\wt G(u)=\int_0^u{\wt g}(s)\,ds$ and
	 $\wt H(u):= \wt g(u)u-2\wt G(u)$ satisfy (A0)--(A5). If, in addition, $g$ satisfies (A4,$\preceq$) or $\frac4NG(u)\preceq H(u)$ for $u\in\R$, then $\wt g$ satisfies the analogous assumptions, and Theorem \ref{th:main} (b) applies to both $g$ and $\wt g$ provided that $N\in\{3,4\}$.
	\item[(E2)] Let $M>0$ and consider a sequence of disjoint and closed intervals $(I_j)_{j=1}^\infty$ in $(0,M)$ such that $\sup I_{j+1} < \inf I_j$ for all $j \geq 1$. Take any decreasing sequence of positive numbers $(a_j)_{j=1}^\infty$ and we define $g'(s)=a_j|s|^{2_*-2}$ for $|s|\in I_j$, $j\geq 1$ and $g'(s)=C |s|^{p-2}$ for $|s|\geq M$, $2_* < p < 2^*$ and properly chosen $C > 0$. We extend $g'(s)/|s|^{2_*-2}$ linearly on $\R$ to a continuous and even function. Note that (A0)--(A3), (A4,$\preceq$) and (A5,$\preceq$) are satisfied with $\eta=(2_*(2_*-1))^{-1}\lim_{j\to\infty} a_j$.
	\item[(E3)] Suppose that $g$ satisfies (A0)--(A5) and $g$ is odd. Similarly as in (E1) we find an interval $[a,b]\subset (0,\infty)$ such that $g'(\zeta) > 0$ for $\zeta\in [a,b]$. Assume for simplicity that $a = 1$. We define $\widetilde{g} : \R \rightarrow \R$ such that $\widetilde{g}(0)=0$ and
	$$
	\widetilde{g}'(s) = \left\{ \begin{array}{ll}
	g'(s) & |s|< 1,\\
	g'(1) |s|^{2_*-2} & 1\leq |s| \leq b \\
	\frac{|b|^{2_*-2}g'(1)}{g'(b)} g'(s) & |s|>b,
	\end{array} \right.
	$$
	Then $\wt G(u)=\int_0^u{\wt g}(s)\,ds$ and
	$\wt H(u):= \wt g(u)u-2\wt G(u)$ satisfy (A0)--(A5). If, in addition,  (A4,$\preceq$) or (A5,$\preceq$) holds for $g$, then the same condition holds for $\wt g$. 
	\item[(E4)] Suppose that $g$ satisfies (A0)--(A5) with some $\eta$ in (A1).
Then $\wt G(u)=\mu |u|^{2_*}+G(u)$, $\mu\geq 0$ and
$\wt H(u):= \wt g(u)u-2\wt G(u)$ satisfy (A0)--(A5) with $\mu+\eta$ in (A1).
If, in addition,  (A4,$\preceq$) or (A5,$\preceq$) holds for $g$, then the same condition holds for $\wt g$. In particular, we can deal with $\mu |s|^{2_*-2}u + |s|^{p-2}u$, $2_* < p < 2^*$ as in \cite[Theorem 1.6]{Soave}.
\end{itemize}

Now we sketch our strategy to find normalized ground state solutions to \eqref{eq}. We believe that the following procedure can be applied to similar variational problems with different differential operators. Contrary to previous works we consider the minimization problem on the closed $L^2$-ball in $H^1(\R^N)$ of radius $\rho$ (instead of the sphere $\cS$) intersected with $\cM$.
\begin{itemize}
	\item[{\em Step 1.}] We show that $J$ is bounded away from $0$ on $\cD\cap \cM$. Here the Gagliardo-Nirenberg inequality \eqref{eqGN} as well as \eqref{eq:Hstrict} play an important role.
	\item[{\em Step 2.}] $J$ is coercive on $\cD\cap \cM$. Here (A4) and the weak monotonicity of $H(u)/|u|^{2_*}$ is important. We adapt some ideas of \cite{SzulkinWeth, JeanjeanLuNorm}, however we do not require the existence of the continuous projection of $H^1(\R^N)\setminus\{0\}$ onto $\cM$ preserving the $L^2$-norm, so the argument is more delicate.
	\item[{\em Step 3.}] If $(u_n)\subset \cD\cap \cM$ is a minimizing sequence, then by means of the profile decomposition Theorem \ref{ThGerard} (\cite[Theorem 1.4]{MederskiBLproblem}) we may find a sequence of translations $(y_n)\subset \R^N$ such that $u_n(\cdot+y_n)$ weakly and a.e. converges to a minimizer $u$ of $J$ on $\cD\cap \cM$. Here a standard one-step concentration-compactness approach in the spirit of Lions \cite{Lions84}  seem to be insufficient, since $u$ may be outside $\cM$. We need to find a full (possibly infinite) decomposition of $(u_n)$ in order to find a weak limit point in $\cM$ up to a proper translations. If $g$ is odd, then working on the ball $\cD$ allows us to use easily the Schwartz symmetrization and we infer that we may find nonnegative and radially symmetric minimizer. The symmetrization approach directly on $\cS\cap\cM$ seems to be cumbersome even for the simplest particular nonlinearities \eqref{Ex:powernon} as in \cite{BartschSoaveJFA,BartschSoaveJFACorr,Soave,Bellazzini,BartschJS2016}.
	\item[{\em Step 4.}] Next we show that for $v \in (\cD \setminus \cS) \cap \cM$ the following crucial inequality holds
	\begin{equation}\label{eq:SD}
	\inf_{\cS\cap \cM} J < J(v),
	\end{equation}
	thus  the minimizer $u$ of $J$ on $\cD\cap\cM$ is attained in fact in  $\cS\cap\cM$.
	\item[{\em Step 5.}] Analysis of Lagrange multipliers $\lambda$ and $\mu$ for constraints $\cS$ and $\cM$ respectively, implies that $\mu=0$ and we conclude that $u$ is a normalized ground state solution to \eqref{eq}.
\end{itemize}

Observe that, an important consequence of {\em Step 4} is that the {\em ground state energy map}
\begin{equation}\label{gsm}
\rho\mapsto \inf_{\cS\cap\cM}J
\end{equation}
defined for $\rho>0$ satisfying \eqref{eq:Hstrict} is strictly decreasing. For a particular power-type nonlinearity and the Schr\"odinger-Poisson problem in $\R^3$, the monotonicity has been investigated in \cite{BellazziniJeanjean}. For more general nonlinearity in \cite{JeanjeanLuNorm} the authors also proved the strict monotonicity of the ground state energy map. Their proof is technical and uses the existence of the continuous  projection of $H^1(\R^N)\setminus\{0\}$ onto $\cM$ preserving the $L^2$-norm, which seems to be not present in our situation. The crucial inequality \eqref{eq:Hstrict} provides the strict monotonicity immediately.  
In Proposition \ref{prop:con} we show also the continuity of the map and the further properties.

\section{Proof of Theorem \ref{th:main}}

Here and in the sequel $C$ denotes a generic positive constant which may vary from one equation to another.

\begin{Lem}\label{lem:ineq}
	Let $f_1,f_2\in \cC(\R)$ such that $f_1(s) \leq f_2(s)$ and $|f_1(s)|+|f_2(s)|\leq C(|s|^2+|s|^{2^*})$ for any $s\in\R$ and some constant $C>0$.
	Then, $f_1(s)\preceq f_2(s)$ for $s\in\R$ if and only if 
	\begin{equation}\label{eq:f1f2}
	\int_{\R^N}f_1(u)-f_2(u)\,dx<0 
	\end{equation}
	for any $u\in H^1(\R^N)\setminus\{0\}$.
\end{Lem}
\begin{proof}
	Suppose that there is a sequence $(s_n)\subset (0,\infty)$ such that $f_1(s_n) < f_2(s_n)$ and $s_n\to 0$ as $n\to\infty$. Note that for any $n$, we find an open interval $I_n\subset (0 ,1/n)$ such that $f_1(s)<f_2(s)$ for $s\in I_n$. We may assume that $I_n$ are pairwise disjoint. Fix $u\in H^1(\R^N)\setminus\{0\}$ and let
	$$\Om:=\Big\{x\in \esssupp u: |u(x)|\in \bigcup_{n\geq 1}I_n\Big\}$$
	and suppose that $|\Om|=0$. Then 
	$$(0,\infty)\setminus  \bigcup_{n\geq 1}I_n=\bigcup_{n\geq 1}J_n$$ is a union of closed and disjoint intervals $J_n$. Note that
	$\big\{x\in \esssupp u: |u(x)|\in J_{n_0}\big\}$ has a positive measure for some $n\geq 1$. We choose $J_{n_0}$ with the largest left endpoint $a:=\inf J_{n_0}$ such that 	$\Omega':=\big\{x\in \esssupp u: |u(x)|\in J_n\big\}$  has a positive measure.
	 Let $b:=\sup\{s< a: s\in\bigcup_{n\geq 1}J_n\}$. Observe that $0<b<a$ and 
	 \begin{equation}\label{eqchi}
	 \int_{\R^N}|u(x+h)-u(x)|^2\,dx\geq (a-b)^2 \int_{\R^N}|\chi_{\Om'}(x+h)-\chi_{\Om'}(x)|^2\,dx
	 \end{equation}
	for a.e. $h\in\R^N$,
	where $\chi_{\Om'}$ is the characteristic function of $\Om'$. Indeed, note that $$|\chi_{\Om'}(x+h)-\chi_{\Om'}(x)|>0$$ if and only if $x+h\in \Om'$ and $x\notin\Omega'$, or $x+h\notin \Om'$ and $x\in\Omega'$. If the latter conditions hold, then $|u(x+h)|\geq a$ and $|u(x)|\leq b$, or $|u(x+h)|\leq b$ and $|u(x)|\geq a$. Then we obtain \eqref{eqchi}.	
	In view of \cite{Ziemer}[Theorem 2.1.6] we infer that $\chi_{\Om'}\in H^1(\R^N)$ and we get the contradiction, thus $|\Om|>0$. Therefore
	$$
	\int_{\R^N} f_1 (u) \, dx = \int_{\R^N \setminus \Om} f_1 (u) \, dx + \int_\Om f_1 (u) \, dx < \int_{\R^N \setminus \Om} f_2 (u) \, dx + \int_\Om f_2 (u) \, dx = \int_{\R^N} f_2 (u) \, dx.
	$$
	Now suppose that there is a sequence $(s_n)\subset (-\infty,0)$ such that $f_1(s_n) < f_2(s_n)$ and $s_n\to 0$ as $n\to\infty$. Then $f_1(-(-s_n)) < f_2(-(-s_n))$, $(-s_n)\subset (0,\infty)$ and by the above proof applied to $f_1(-\cdot)$, $f_2(-\cdot)$ and $-u$ we obtain  
	$$
	\int_{\R^N} f_1 (u) \, dx=	\int_{\R^N} f_1 (-(-u)) \, dx < 	\int_{\R^N} f_2 (-(-u)) \, dx =\int_{\R^N} f_2 (u) \, dx.
	$$
	Therefore \eqref{eq:f1f2} holds provided that  $f_1(s)\preceq f_2(s)$ for $s\in\R$.
	
	On the other hand, suppose by contradiction that \eqref{eq:f1f2} holds for every $u \in H^1 (\R^N)\setminus\{0\}$, $f_1(s) \leq f_2(s)$ for all $s \in \R$ and there is an open interval $I \subset \R$ such that $0 \in \overline{I}$ and $f_1(s) = f_2(s)$ on $I$. We may assume that $\sup I>0$. Take $a>0$  such that $a \in I$ and let
	$$
	\varphi(x) := a \exp \left( - \frac{|x|^2}{1-|x|^2} \right) \chi_{[0,1)}(|x|), \quad x \in \R^N.
	$$
	Then $\varphi \in \cC_0^\infty (\R^N) \subset H^1 (\R^N)$ is such that $\varphi (\R^N) \subset \overline{I}$. Hence $f_1(\varphi(x))=f_2(\varphi(x))$ for all $x \in \R^N$ and
	$$
	\int_{\R^N} f_1 (\varphi) - f_2(\varphi) \, dx  = 0,
	$$
	and we obtain a contradiction with \eqref{eq:f1f2}.
\end{proof}

In view of (A6) and arguing as in \cite[page 325]{BerLions}, for any $R > 0$ one can find a radial function $u \in H^1_0 (B(0,R)) \cap L^\infty (B(0,R))$ such that $\int_{\R^N} H(u) \, dx > 0$. Then $u(r(u)\cdot)\in\cM$ for
\begin{equation}\label{def:r}
r(u) := \left( \frac{\frac{N}{2} \int_{\R^N} H(u) \, dx}{\int_{\R^N} |\nabla u|^2 \, dx} \right)^{1/2},
\end{equation}
so that $\cM$ is nonempty.

\begin{Lem}\label{lem1} 
	Suppose that (A0), (A1), (A3), (A5), (A6) and \eqref{eq:Hstrict} are satisfied. There holds
	$$
	\inf_{u \in \cD \cap \cM} |\nabla u|_2 > 0.
	$$
\end{Lem}

\begin{proof} Take any $2+\frac{4}{N} < p < 2^*$.
In view of (A1), (A3) and (A5) for any $\eps>0$ there is $c_\eps>0$ such that
\begin{equation*}
	H(u) \leq (2^*-2)G(u) \leq (2^*-2) \left( \varepsilon |u|^{2^*} + (\eps+\eta) |u|^{2+\frac4N}  + c_\varepsilon |u|^p \right)
\end{equation*}
for any $u\in\R$. From the Gagliardo-Nirenberg inequality
	$$
	|u|_p \leq C_{N,p} |\nabla u|_2^\delta |u|_2^{1-\delta} \leq C_{N,p} \rho^{\frac{1-\delta}{2}} |\nabla u|_2^\delta,
	$$
	where $\delta = N \big( \frac{1}{2} - \frac{1}{p} \big)$. Note that
	$$
	\delta p =N \left( \frac{p}{2} - 1 \right) > N \left( 1+\frac{2}{N}-1\right) = 2.
	$$
	Since $u \in \cD \cap \cM$, we get
	\begin{eqnarray*}
	|\nabla u|_2^2 = \frac{N}{2}   \int_{\R^N} H(u) \, dx &\leq& \frac{N}{2} (2^*-2) \left( \varepsilon \left( |u|_{2^*}^{2^*} + |u|_{2+\frac4N}^{2+\frac4N} \right) +\eta  |u|_{2+\frac4N}^{2+\frac4N}+  c_\varepsilon C^p_{N,p}  \rho^{\frac{(1-\delta)p}{2}}  |\nabla u|_2^{\delta p} \right)\\
	&=& 2^* \left( \varepsilon \left( |u|_{2^*}^{2^*} + |u|_{2+\frac4N}^{2+\frac4N} \right) +\eta  |u|_{2+\frac4N}^{2+\frac4N}+  c_\varepsilon C^p_{N,p}  \rho^{\frac{(1-\delta)p}{2}}  |\nabla u|_2^{\delta p} \right) \\
	&\leq& \varepsilon C\big(
	|\nabla u|_2^{2^*}+|\nabla u|_2^{2}\big)
	+Cc_\eps|\nabla u|_2^{\delta p}+2^*\eta C_{N,2_*}^{2_*} \rho^{\frac{2}{N}} |\nabla u|_2^2 \\
	&=& \varepsilon C | \nabla u|_2^{2^*} + C c_\eps |\nabla u|_2^{\delta p} + \left( \eps C + 2^* \eta C_{N,2_*}^{2_*} \rho^{\frac{2}{N}} \right) |\nabla u|_2^2
	\end{eqnarray*}
	for a constant $C>0$. Taking $\eps <\frac{1}{C} \left(1-2^* \eta C_{N,2_*}^{2_*} \rho^{\frac{2}{N}} \right)$ we obtain that
	$|\nabla u|_2^2$ is bounded away from $0$ on $\cD\cap\cM$ provided that 
	$$2^*\eta C_{N,2_*}^{2_*} \rho^{\frac{2}{N}} <1.$$

\end{proof}

Let $u\in H^1(\R^N)\setminus\{0\}$ be such that
\begin{equation}\label{ineq:2}
2 \eta C_{N,2_*}^{2_*} \left( \int_{\R^N} |u|^2 \, dx \right)^{2/N} < 1.
\end{equation}
Define $$\vp(\lambda):=J(\lambda^{\frac{N}{2}} u(\lambda \cdot)),\quad \lambda\in (0,\infty).$$ 
In particular we can consider $u \in \cD$ such that \eqref{eq:Hstrict} holds or $u \in H^1 (\R^N) \setminus \{0\}$ if $\eta = 0$.

\begin{Lem}\label{lem:phi}
Suppose that $u\in H^1(\R^N)\setminus\{0\}$ satisfies \eqref{ineq:2}. Assume moreover that (A0), (A1), (A3)--(A6) hold. Then there is an interval $[a,b]\subset (0,\infty)$ such that $\vp$ is constant 
on $[a,b]$, $\lambda^{\frac{N}{2}} u(\lambda \cdot)\in\cM$  and $\vp(\lambda)\geq \vp(\lambda')$ for any $\lambda\in [a,b]$ and $\lambda'\in (0,\infty)$ and the strict inequality holds for  $\lambda'\in (0,\infty)\setminus [a,b]$. Moreover if $u\in\cM$, then $1\in [a,b]$.
\end{Lem}
\begin{proof}
Fix $u \in H^1(\R^N)\setminus\{0\}$ satisfying \eqref{ineq:2}. Observe that, from (A1),
$$
\varphi(\lambda) = \frac{\lambda^2}{2} \int_{\R^N} |\nabla  u|^2 \, dx - \int_{\R^N} \frac{G(\lambda^{N/2}u)}{\lambda^N} \, dx  \to 0
$$
as $\lambda \to 0^+$. Let $R := |u|_2^2 = \left| \lambda^{N/2} u(\lambda \cdot) \right|_2^2 > 0$. Note that, from (A1), (A3), (A5) and the Gagliardo-Nirenberg inequality, for every $\varepsilon > 0$ there is $c_\eps > 0$ such that
\begin{equation*}
	\int_{\R^N} G(u) \, dx \leq  (\eps + \eta) |u|^{2+\frac4N}_{2+\frac4N} + c_\eps |u|^{2^*}_{2^*} \leq (\eps + \eta) C_{N,2_*}^{2_*} |\nabla u|_2^2 R^{\frac2N} + C_\eps C |\nabla u|_2^{2^*}.
\end{equation*}
Hence
\begin{align*}
\frac{\varphi(\lambda)}{\lambda^2} &= \frac{1}{2} \int_{\R^N} |\nabla  u|^2 \, dx - \frac{1}{\lambda^2} \int_{\R^N} G \left( \lambda^{\frac{N}{2}} u(\lambda x) \right) \, dx \\
&\geq \frac12 |\nabla u|_2^2 - (\eps + \eta) C_{N,2_*}^{2_*} |\nabla u|_2^2 R^{\frac2N} - C_\eps C \lambda^{2^*-2} |\nabla u|_2^{2^*}\\
& = \frac{|\nabla u|_2^2}{2} \left(1 - 2 (\eta+\eps) C_{N,2_*}^{2_*} R^{\frac2N} \right) + o(1)
\end{align*}
as $\lambda \to 0^+$, and $\varphi(\lambda) > 0$ for sufficiently small $\lambda > 0$.
Moreover from (A2) there follows that
$$
\frac{\varphi(\lambda)}{\lambda^2} = \frac12 \int_{\R^N} |\nabla  u|^2 \, dx - \int_{\R^N} \frac{G(\lambda^{N/2}u)}{ \left(\lambda^{N/2} \right)^{2_*} } \, dx \to -\infty
$$
as $\lambda \to \infty$. Hence $\varphi$ has a maximum at some $\lambda_0 > 0$. In particular $\varphi'(\lambda_0) = 0$, so that
$$0=\vp'(\lambda_0)=\lambda_0 \Big(\int_{\R^N}|\nabla u|^2\,dx-\frac{N}{2}\int_{\R^N}H(\lambda_0^{N/2} u)(\lambda_0^{N/2})^{-2_*}\,dx\Big)$$
and $\int_{\R^N} \left| \nabla \lambda_0^{\frac{N}{2}} u(\lambda_0 \cdot) \right|^2 \, dx = \frac{N}{2} \int_{\R^N} H\left( \lambda_0^{\frac{N}{2}} u(\lambda_0 \cdot) \right) \, dx.$ Hence $\lambda_0^{\frac{N}{2}} u(\lambda_0 \cdot) \in \cM$. From (A4) there follows that the function
$$
(0, \infty) \ni \lambda \mapsto \int_{\R^N} H(\lambda^{N/2} u)(\lambda^{N/2})^{-2_*} \, dx \in \R
$$
is nonincreasing. Moreover, from (A2) and (A5), $\int_{\R^N} H(\lambda^{N/2} u)(\lambda^{N/2})^{-2_*} \, dx \to \infty$ as $\lambda \to \infty$. Hence there is an inteval $[a,b]$ such that $\varphi'(\lambda) = 0$ for $\lambda \in [a,b]$. In particular, $\lambda^{\frac{N}{2}} u(\lambda \cdot)\in\cM$  and $\vp(\lambda)\geq \vp(\lambda')$ for any $\lambda\in [a,b]$ and $\lambda'\in (0,\infty)$ and the strict inequality holds for $\lambda'\in (0,\infty)\setminus [a,b]$. Since $\varphi'(\lambda) = 0$ for $\lambda \in [a,b]$ then $\vp$ is constant on $[a,b]$. If, in addition, $u \in \cM$, then
$$
\int_{\R^N} |\nabla u|^2 \, dx = \frac{N}{2} \int_{\R^N} H(u) \, dx,
$$
so that $\varphi'(1) = 0$ and $1 \in [a,b]$.
\end{proof}

\begin{Lem}\label{lemCoercive}
Suppose that (A0)--(A5) and \eqref{eq:Hstrict} hold. Then $J$ is coercive on $\cD \cap \cM$.
\end{Lem}

\begin{proof}
Observe that for $u \in \cD \cap \cM$, taking (A5) into account, we have
$$
J(u) = J(u) - \frac12 M(u) = \frac{N}{4} \int_{\R^N} H(u) - \frac{4}{N} G(u) \, dx \geq 0.
$$
Hence $J$ is bounded from below on $\cD \cap \cM$.
Now we follow similar arguments as in \cite[Lemma 2.5]{JeanjeanLuNorm},  \cite[Proposition 2.7]{SzulkinWeth}.
Suppose that $(u_n) \subset\cD \cap \cM$ is a sequence such that $\|u_n\| \to \infty$ and $J(u_n)$ is bounded from above. Since $u_n \in \cD$ we see that $|\nabla u_n|_2^2 \to \infty$. Put
$$
\lambda_n := \frac{1}{|\nabla u_n|_2} > 0
$$
and define
$$
v_n := \lambda_n^{N/2} u_n \left( \lambda_n \cdot \right).
$$
Note that $\lambda_n \to 0^+$ as $n \to \infty$. Then 
$$
\int_{\R^N} |v_n|^2 \, dx = \int_{\R^N} |u_n|^2 \, dx \leq \rho,
$$
so that $v_n \in \cD$. Moreover
$$
|\nabla v_n |_2^2 = \int_{\R^N} |\nabla v_n|^2 \, dx = \lambda_n^{N} \lambda_n^{-N+2} \int_{\R^N} |\nabla u_n|^2 \, dx = \lambda_n^{2} \lambda_n^{-2} = 1.
$$
In particular, $(v_n)$ is bounded in $H^1 (\R^N)$. Suppose that 
$$
\limsup_{n \to \infty} \left( \sup_{y \in \R^N} \int_{B(y,1)} |v_n|^2 \, dx \right) > 0.
$$
Then, up to a subsequence, we can find translations $(z_n) \subset \R^N$ such that
$$
v_n (\cdot + z_n) \weakto v \neq 0 \hbox{ in } H^1 (\R^N)
$$
and $v_n(x+z_n) \to v(x)$ for a.e. $x \in \R^N$. Then by (A2)
\begin{align*}
0 \leq \frac{J(u_n)}{|\nabla u_n|_2^2} &= \frac12 - \int_{\R^N} \frac{G(u_n)}{|\nabla u_n|_2^2} \, dx = \frac12 - \lambda_n^{N} \lambda_n^2 \int_{\R^N} G(u_n(\lambda_n x)) \, dx \\
&= \frac12 - \lambda_n^{N+2} \int_{\R^N} G( \lambda_n^{-N/2} v_n ) \, dx = \frac12 - \lambda_n^{N+2} \int_{\R^N} \frac{G(\lambda_n^{-N/2} v_n)}{ \left| \lambda_n^{-N/2} v_n \right|^{2+ \frac{4}{N}} } \left| \lambda_n^{-N/2} v_n \right|^{2+ \frac{4}{N}} \, dx \\
&= \frac12 -  \int_{\R^N} \frac{G(\lambda_n^{-N/2} v_n)}{ \left| \lambda_n^{-N/2} v_n \right|^{2+ \frac{4}{N}} } \left|  v_n \right|^{2+ \frac{4}{N}} \, dx \\
&= \frac12 -  \int_{\R^N} \frac{G(\lambda_n^{-N/2} v_n(x+z_n))}{ \left| \lambda_n^{-N/2} v_n(x+z_n) \right|^{2+ \frac{4}{N}} } \left|  v_n(x+z_n) \right|^{2+ \frac{4}{N}} \, dx  \to -\infty
\end{align*}
and we obtain a contradiction. Hence we may assume that
$$
\sup_{y \in \R^N} \int_{B(y,1)} |v_n|^2 \, dx \to 0
$$
and from Lion's lemma $v_n \to 0$ in $L^{2 + \frac{4}{N}} (\R^N)$. Observe that 
$$
u_n = \lambda_n^{-N/2} v_n \left( \frac{\cdot}{\lambda_n} \right) \in \cM.
$$
Since $u_n = \lambda_n^{-N/2} v_n \left( \frac{\cdot}{\lambda_n} \right) \in \cD$ and \eqref{eq:Hstrict} holds, $u_n$ satisfies also \eqref{ineq:2}. Hence, from Lemma \ref{lem:phi}, for any $\lambda >0$ there holds
\begin{align*}
J(u_n) = J \left( \lambda_n^{-N/2} v_n \left( \frac{\cdot}{\lambda_n} \right) \right) \geq J \left( \lambda^{N/2} v_n (\lambda \cdot) \right) = \frac{\lambda^2}{2} - \lambda^{-N} \int_{\R^N} G \left( \lambda^{N/2} v_n \right) \, dx
\end{align*}
and $\lambda^{-N} \int_{\R^N} G \left( \lambda^{N/2} v_n \right) \, dx \to 0$ as $n \to \infty$. Thus we obtain a contradiction by taking sufficiently large $\lambda > 0$.
\end{proof}

\begin{Lem} Suppose that (A0), (A1), (A3)--(A6) and \eqref{eq:Hstrict} hold. There holds
	$$
	c:=\inf_{\cD \cap \cM} J > 0
	$$
\end{Lem}

\begin{proof}
We will show that for $\rho > 0$ satisfying \eqref{eq:Hstrict} there is $\delta > 0$ such that 
\begin{equation}\label{eq:JJ}
\frac{1}{2N} |\nabla u|^2_2  \leq J(u)
\end{equation}
for $u \in \cD$ such that $|\nabla u|_2 \leq \delta$. From the Gagliardo-Nirenberg inequality we obtain
\begin{align*}
\int_{\R^N} G(u) \, dx &\leq (\eps + \eta) |u|_{2+\frac{4}{N}}^{2+\frac{4}{N}} + C_\eps |u|_{2^*}^{2^*} \leq (\eps+\eta) C_{N,2_*}^{2_*} \rho^{\frac{2}{N}} |\nabla u|_2^2 + C_\eps C_{N,2^*}^{2^*} |\nabla u|_2^{2^*} \\
&= \left( \eps C_{N,2_*}^{2_*} \rho^{\frac{2}{N}} + C_\eps C_{N,2^*}^{2^*} |\nabla u|_2^{\frac{4}{N-2}} + \eta C_{N,2_*}^{2_*} \rho^{\frac{2}{N}} \right) |\nabla u|_2^2 \\
&< \left( \eps C_{N,2_*}^{2_*} \rho^{\frac{2}{N}} + C_\eps C_{N,2^*}^{2^*} |\nabla u|_2^{\frac{4}{N-2}} + \frac{1}{2^*} \right) |\nabla u|_2^2 \\
&= \left( \eps C_{N,2_*}^{2_*} \rho^{\frac{2}{N}} + C_\eps C_{N,2^*}^{2^*} |\nabla u|_2^{\frac{4}{N-2}} + \frac{1}{2} - \frac{1}{N} \right) |\nabla u|_2^2.
\end{align*}
Taking
$$
\eps := \frac{1}{4N C_{N,2_*}^{2_*} \rho^{\frac{2}{N}}} > 0, \quad \delta := \left( \frac{1}{4N C_\eps C_{N,2^*}^{2^*}} \right)^{\frac{N-2}{4}} > 0
$$
we obtain that
$$
\int_{\R^N} G(u) \, dx \leq \left( \frac{1}{4N} + \frac{1}{4N} + \frac12 - \frac{1}{N} \right) |\nabla u|_2^2 = \left( \frac12 - \frac{1}{2N} \right) |\nabla u|_2^2.
$$
Hence
$$
J(u) = \frac12 |\nabla u|_2^2 - \int_{\R^N} G(u) \, dx \geq \frac12 |\nabla u|_2^2 - \left( \frac12 - \frac{1}{2N} \right) |\nabla u|_2^2 = \frac{1}{2N} |\nabla u|_2^2.
$$
Fix $u \in \cD \cap \cM$. In view of \eqref{eq:Hstrict}, $u$ clearly satisfies the inequality \eqref{ineq:2}. Then, from Lemma \ref{lem:phi}, for every $\lambda > 0$ there holds
$$
J(u) \geq J(\lambda^{N/2} u(\lambda \cdot)).
$$
Choose $\lambda := \frac{\delta}{|\nabla u|_2} > 0$, where $\delta > 0$ is chosen so that \eqref{eq:JJ} holds, and let $v :=  \lambda^{N/2} u(\lambda \cdot)$. Obviously $|v|_2 = |u|_2$ so that $v \in \cD$. Moreover $|\nabla v|_2 = \delta$. Then
$$
J(u) \geq J(v) \geq \frac{1}{2N} |\nabla v|_2^2 = \frac{1}{2N} \delta^2 >0.
$$
\end{proof}

Before we show that $\inf_{\cD \cap \cM} J$ is attained, we need the following profile decomposition result obtained in \cite[Theorem 1.4]{MederskiBLproblem} applied to $H$ satisfying
$$\lim_{u\to 0} H(u)/|u|^{2}=\lim_{|u|\to\infty}H(u)/|u|^{2^*}=0.$$

\begin{Th}\label{ThGerard}
	Suppose that $(u_n)\subset H^{1}(\R^N)$ is bounded.
	Then there are sequences
	$(\tu_i)_{i=0}^\infty\subset H^1(\R^N)$, $(y_n^i)_{i=0}^\infty\subset \R^N$ for any $n\geq 1$, such that $y_n^0=0$,
	$|y_n^i-y_n^j|\rightarrow \infty$ as $n\to\infty$ for $i\neq j$, and passing to a subsequence, the following conditions hold for any $i\geq 0$:
	\begin{eqnarray}
	\nonumber
	&& u_n(\cdot+y_n^i)\weakto \tu_i\; \hbox{ in } H^1(\R^N)\text{ as }n\to\infty,\\
	\label{EqSplit2a}
	&& \lim_{n\to\infty}\int_{\R^N}|\nabla u_n|^2\, dx=\sum_{j=0}^i \int_{\R^N}|\nabla\tu_j|^2\, dx+\lim_{n\to\infty}\int_{\R^N}|\nabla v_n^i|^2\, dx,
	\end{eqnarray}
	where $v_n^i:=u_n-\sum_{j=0}^i\tu_j(\cdot-y_n^j)$ and
	\begin{eqnarray}
	&&	 \limsup_{n\to\infty}\int_{\R^N}H(u_n)\, dx= \sum_{j=0}^\infty
	\int_{\R^N}H(\tu_j)\, dx.	\label{EqSplit3a}
	\end{eqnarray}
\end{Th}

\begin{Lem}\label{lem:c_0attained} 
Suppose that (A0)--(A5) and \eqref{eq:Hstrict} hold. Then $c=\inf_{\cD \cap \cM} J$ is attained. If, in addition, $g$ is odd, then $c$ is attained by a nonnegative and radially symmetric function in $\cD \cap \cM$.
\end{Lem}

\begin{proof}
Take any sequence $(u_n) \subset \cD \cap \cM$ such that $J(u_n) \to c$ and by Lemma \ref{lemCoercive}, $(u_n)$ is bounded in $H^1(\R^N)$. Note that by (A1), (A3) and (A5), we may apply Theorem \ref{ThGerard} we find a profile decomposition of $(u_n)$ satisfying \eqref{EqSplit2a} and \eqref{EqSplit3a}.
We show that
$$0<\int_{\R^N}|\nabla \tu_i|^2\,dx\leq \frac{N}{2}\int_{\R^N} H(\tu_i) \, dx$$
for some $i\geq 0$. Let
$$I:=\{i\geq 0: \tu_i\neq 0\}.$$
In view of Lemma \ref{lem1} and \eqref{EqSplit3a}, $I\neq \emptyset$.
Suppose that  
$$\int_{\R^N}|\nabla \tu_i|^2\,dx>\frac{N}{2}\int_{\R^N} H(\tu_i) \, dx$$
for all $i\in I$.
Then by \eqref{EqSplit2a} and \eqref{EqSplit3a}
\begin{eqnarray*}
\limsup_{n\to\infty} \frac{N}{2}\int_{\R^N}H(u_n)\,dx&=&\limsup_{n\to\infty}\int_{\R^N}|\nabla u_n|^2\,dx
\geq \sum_{j=0}^\infty \int_{\R^N}|\nabla\tu_j|^2\, dx=\sum_{j\in I} \int_{\R^N}|\nabla\tu_j|^2\, dx\\
&>& \sum_{j=0}^\infty \frac{N}{2}\int_{\R^N} H(\tu_j) \, dx=\limsup_{n\to\infty} \frac{N}{2}\int_{\R^N}H(u_n)\,dx,
\end{eqnarray*}
which is a contradiction. 
Therefore there is $i\in I$ such that $r(\tu_i)\geq 1$ defined as in \eqref{def:r} and
$\tu_i(r(\tu_i) \cdot) \in \cM$.  Moreover
	$$\int_{\R^N}|\tu_i(r(\tu_i)\cdot)|^2\,dx=r(\tu_i)^{-N}\int_{\R^N}|\tu_i|^2\,dx\leq r(\tu_i)^{-N}\rho\leq \rho,$$
hence  $\tu_i(r(\tu_i)\cdot)\in\cD\cap\cM$.
If $r(\tu_i)>1$, then passing to a subsequence $u_n(x+y_n^i)\to\tu_i(x)$ for a.e. $x\in\R^N$ and by Fatou's lemma
	\begin{eqnarray*}
		0<\inf_{\cD\cap\cM}J
		&\leq& J(\tu_i(r(\tu_i)\cdot))
		=r(\tu_i)^{-N}\frac{N}{4}\int_{\R^N}H(\tu_i)-\frac{4}{N}G(\tu_i)\, dx\\
		&<& \frac{N}{4}\int_{\R^N}H(\tu_i)-\frac{4}{N}G(\tu_i)\, dx\\
		&\leq& \liminf_{n\to\infty} \frac{N}{4}\int_{\R^N}H(u_n(\cdot+y_n^i))-\frac{4}{N}G(u_n(\cdot+y_n^i))\, dx\\
		&=& \liminf_{n\to\infty} J(u_n)=c=\inf_{\cD\cap\cM}J
	\end{eqnarray*}
and again we get a contradiction. Therefore
	$r(\tu_i)=1$, $\tu_i\in\cD\cap\cM$  and 
	\begin{eqnarray*}
		J(\tu_i)
		&=&\frac{N}{4}\int_{\R^N}H(\tu_i)-\frac{4}{N}G(\tu_i)\, dx\\
		&\leq& \liminf_{n\to\infty} \frac{N}{4}\int_{\R^N}H(u_n(\cdot+y_n^i))-\frac{4}{N}G(u_n(\cdot+y_n^i))\, dx\\
		&=& \liminf_{n\to\infty} J(u_n)=c.
	\end{eqnarray*}
Thus $J(\tu_i)=c$.

Suppose that $g$ is odd. Then $G$ and $H$ are even, so that $G(|u|)=G(u)$ and $H(|u|)=H(u)$ for all $u \in H^1(\R^N)$. We define
$
\tv_i := |\tu_i|^*
$
as the Schwarz symmetrization of $|\tu_i|$.
Then $|\tv_i|_2 = |\tu_i|_2$, hence $\tv_i \in \cD$. Moreover, since
$$
\int_{\R^N} |\nabla \tv_i|^2 \, dx \leq \int_{\R^N} |\nabla \tu_i|^2 \, dx = \frac{N}{2} \int_{\R^N} H(\tu_i) \, dx = \frac{N}{2} \int_{\R^N} H(\tv_i) \, dx,
$$
we obtain that $r(\tv_i) \geq 1$, where $r$ is given by \eqref{def:r} and $\tv_i(r(\tv_i) \cdot) \in \cM$. Suppose that $r(\tv_i) > 1$. Then
\begin{align*}
\inf_{\cD \cap \cM} J &\leq J(\tv_i (r(\tv_i) \cdot)) = r(\tv_i)^{-N} \frac{N}{4} \int_{\R^N} H(\tv_i) - \frac{4}{N} G(\tv_i) \, dx \\
&< \frac{N}{4} \int_{\R^N} H(\tv_i) - \frac{4}{N} G(\tv_i) \, dx = \frac{N}{4} \int_{\R^N} H(\tu_i) - \frac{4}{N} G(\tu_i) \, dx = J(\tu_i) = \inf_{\cD \cap \cM} J,
\end{align*}
which is a contradiction. Hence $r(\tv_i) = 1$ and $\tv_i \in \cM$. Obviously $J(\tv_i) = \inf_{\cD \cap \cM} J$, $\tv_i \geq 0$ and $\tv_i$ is radially symmetric. 

\end{proof}

\begin{Lem}\label{lem:SDineq} Suppose that (A0)--(A5), \eqref{eq:Hstrict} hold. Assume moreover that \\
(a) (A5, $\preceq$) hold\\
or\\
(b) $\frac4N G(s) \preceq H(s)$ for $s \in \R$, $g$ is odd and $N \in \{3,4\}$.\\
For any $u \in (\cD \setminus \cS) \cap \cM$ there holds
	$$
	\inf_{\cS \cap \cM} J < J(u).
	$$
\end{Lem}

\begin{proof}
	Suppose by contradiction that there is $\widetilde{u} \in \cM$ such that $\int_{\R^N} |\widetilde{u}|^2 \, dx < \rho$ and
	$$
	c= J(\widetilde{u}) \leq \inf_{\cS \cap \cM} J.
	$$
	Hence $\widetilde{u}$ is a local minimizer for $J$ on $\cD \cap \cM$. Since $\cD \setminus \cS$ is an open set in $\cM$, we see that $\widetilde{u}$ is a local minimizer of $J$ on $\cM$. Hence there is a Lagrange multiplier $\mu \in \R$ such that
	$$
	J'(\widetilde{u})(v) + \mu \left( \int_{\R^N} \nabla \widetilde{u} \nabla v \, dx - \frac{N}{4} \int_{\R^N} h(\widetilde{u}) v \, dx \right) = 0
	$$
	for any $v \in H^1(\R^N)$, i.e. $\widetilde{u}$ is a weak solution to
	$$
	-\Delta \widetilde{u} - g(\widetilde{u}) + \mu \left( -\Delta \widetilde{u} - \frac{N}{4} h(\widetilde{u}) \right) = 0
	$$
	or equivalently
	$$
	-(1+\mu) \Delta \widetilde{u} = g(\tu) + \frac{N}{4} \mu h(\widetilde{u}).
	$$
	In particular $\widetilde{u}$ satisfies the following Nehari-type identity
	$$
	(1+\mu) \int_{\R^N} |\nabla\widetilde{u}|^2 \, dx = \int_{\R^N} g(\widetilde{u}) \widetilde{u} + \frac{N}{4} \mu h(\widetilde{u})\widetilde{u} \, dx.
	$$
	If $\mu = -1$ we obtain that
	$$
	\int_{\R^N} g(\widetilde{u}) \widetilde{u} - \frac{N}{4} h(\widetilde{u})\widetilde{u} \, dx = 0.
	$$
	On the other hand, by (A4), (A5), $\frac4N G(s) \preceq H(s)$ for $s \in \R$, and Lemma \ref{lem:ineq}
	\begin{align*} 
	\int_{\R^N} g(\widetilde{u}) \widetilde{u} - \frac{N}{4} h(\widetilde{u})\widetilde{u} \, dx &= - \frac{N}{4} \int_{\R^N} h(\widetilde{u})\widetilde{u} - \frac{4}{N} g(\widetilde{u}) \widetilde{u} \, dx \\
	&\leq - \frac{N}{4} \int_{\R^N} \left(2+\frac{4}{N}\right) H(\tu)-\frac{4}{N} g(\tu)\tu \, dx \\
	&= - \frac{N}{2} \int_{\R^N} H(\tu)- \frac{4}{N} G(\tu) \, dx < 0,
	\end{align*} 
	and we obtain a contradiction. Hence $\mu \neq -1$. Since $\widetilde{u} \in \cM$ we obtain
	\begin{align*}
	\int_{\R^N} |\nabla \widetilde{u}|^2 \, dx = \frac{N}{2} \int_{\R^N} H(\widetilde{u}) \, dx.
	\end{align*}
	On the other hand $\widetilde{u}$ satisfies Pohozaev and Nehari identities. Thus
	$$
	(1+\mu) \int_{\R^N} |\nabla \widetilde{u}|^2 \, dx = \frac{N}{2} \int_{\R^N} H(\widetilde{u}) + \frac{N}{4} \mu \left( h(\widetilde{u})\widetilde{u} - 2 H(\widetilde{u}) \right) \, dx.
	$$
	Combining these two identities we get
	$$
	(1+\mu) \frac{N}{2} \int_{\R^N} H(\widetilde{u}) \, dx = \frac{N}{2} \int_{\R^N} H(\widetilde{u}) + \frac{N}{4} \mu \left( h(\widetilde{u})\widetilde{u} - 2 H(\widetilde{u}) \right) \, dx.
	$$
	Thus
	$$
	\mu \int_{\R^N} H(\widetilde{u}) \, dx = \frac{N}{4} \mu \int_{\R^N}  h(\widetilde{u})\widetilde{u} - 2 H(\widetilde{u}) \, dx.
	$$
	If $\mu\neq0$, then
	$$
	\int_{\R^N} h(\widetilde{u})\widetilde{u} - \left(2 + \frac{4}{N} \right) H(\widetilde{u}) \, dx=0.
	$$
	From the elliptic regularity theory we may assume that $\tu$ is continuous and $h(\widetilde{u}(x))\widetilde{u}(x) - \left(2 + \frac{4}{N} \right) H(\widetilde{u}(x)) = 0$ for all $x \in \R^N$. Since $\tu \in H^1 (\R^N)$, we know that $\tu(x) \to 0$ as $|x| \to \infty$. In particular, there is an open interval $I$ such that $0 \in \overline{I}$ and $h(u)u - \left(2 + \frac{4}{N} \right) H(u) = 0$ for $u \in \overline{I}$. Hence $H(u) = C |u|^{2+\frac{4}{N}}$ for some $C > 0$ and $u \in \overline{I}$, which is a contradiction with the first inequality in (A5,$\preceq$). 	
	Thus we have $\mu = 0$ and $\widetilde{u}$ is a weak solution to
	$$
	-\Delta \widetilde{u}  = g(\widetilde{u}).
	$$
	From the Nehari-type identity
	$$
	\int_{\R^N} |\nabla \widetilde{u}|^2 \, dx = \int_{\R^N} g(\widetilde{u}) \widetilde{u} \, dx
	$$
	and since $\widetilde{u} \in \cM$ 
we obtain
	$$
	\int_{\R^N} g(\widetilde{u}) \widetilde{u} \, dx = \frac{N}{2} \int_{\R^N} H(\widetilde{u}) \, dx
	$$
	and
	$$
	\int_{\R^N} 2^* G(\widetilde{u}) -  g(\widetilde{u})\widetilde{u} \, dx = 0.
	$$
By the elliptic regularity theory, $\tu$ is continuous and
in view of (A5)
$$2^* G(\widetilde{u}(x))=g(\widetilde{u}(x))\widetilde{u}(x)$$ for $x\in\R^N$. Since $\tu\in H^1(\R^N)$, there is an open interval $I \subset \R$ such that $0 \in \overline{I}$ and
$2^* G(u)=g(u)u$ for $u\in \overline{I}$.
Then there is $C>0$ such that $G(u)=C|u|^{2^*}$ for $u\in \overline{I}$. Now we need to consider two cases.\\
(a) If the inequality (A5, $\preceq$) holds, then we obtain a contradiction immediately.
(b) If $g$ is odd, then in view of Lemma \ref{lem:c_0attained}, we may assume that $\tu$ is nonnegative and radially symmetric. Moreover $\tu$ solves
\begin{equation}\label{eq:critical}
-\Delta \tu = (2^* C) |\tu|^{2^* - 2 } \tu,
\end{equation}
and we get a contradiction, since the nonnegative and radial solution to problem \eqref{eq:critical} is a Aubin-Talenti instanton, up to a scaling and a translation, which is not $L^2$-integrable if $N \in \{3,4\}$, see \cite[Section 6.2]{delPino}, cf. \cite{Aubin,Talenti}. 

\end{proof}

\begin{altproof}{Theorem \ref{th:main}}
In view of Lemma \ref{lem:c_0attained} and Lemma \ref{lem:SDineq} we infer that $c=\inf_{\cS \cap \cM} J$ is attained.
	Now we find Lagrange multipliers $\lambda,\mu\in\R$ such that $\tu\in\cS\cap\cM$ solves
	$$-\Delta \tu-g(\tu)+\lambda \tu +\mu \Big(-\Delta \tu-\frac{N}{4}h(\tu)\Big)=0,$$
	that is
	\begin{equation}\label{eq:mu}
	-(1+\mu)\Delta \tu +\lambda\tu =g(\tu)+\frac{N}{4}\mu h(\tu).
	\end{equation}
	Suppose that $\mu=-1$ and consider two cases.\\
(a)
Suppose that $g'(u)=o(1)$ as $u \to 0$. Then, from (A1) and (A5) there follows that $g(u)=o(u)$ and $h(u)=o(u)$ as $u \to 0$. Note that by (A4), the first inequality in (A5, $\preceq$) and Lemma \ref{lem:ineq}
	\begin{align*}
	\lambda\int_{\R^N}|\tu|^2\,dx &=\int_{\R^N}g(\tu)\tu+\frac{N}{4}\mu h(\tu)\tu\, dx = \frac{N}{4} \int_{\R^N}\frac{4}{N} g(\tu)\tu- h(\tu)\tu\, dx \\
	&\leq \frac{N}{2} \int_{\R^N} \frac{4}{N} G(\tu) - H(\tu) \, dx < 0,
	\end{align*}
	hence $\lambda<0$.
On the other hand, take
$$\Sigma:=\Big\{x\in\R^N: \lambda\tu(x) =g(\tu(x))-\frac{N}{4}h(\tu(x))\Big\}$$ 
and note that the measure of
$\Om:=\{x\in \Sigma: \tu(x)\neq 0\}$ is nonzero.
Suppose that
$\delta:=\essinf_{x\in\Om}|\tu(x)|>0$. Since $\tu\in L^2(\R^N)\setminus\{0\}$, we infer that $\Om$ has finite positive measure and observe that
$$\int_{\R^N}|\tu(x+h)-\tu(x)|^2\,dx\geq \delta^2 \int_{\R^N}|\chi_\Om(x+h)-\chi_\Om(x)|^2\,dx\quad\hbox{ for any }h\in\R^N,$$
where $\chi_\Om$ is the characteristic function of $\Om$. In view of \cite[Theorem 2.1.6]{Ziemer} we infer that $\chi_\Om\in H^1(\R^N)$, hence  we get a contradiction. Therefore we find a sequence $(x_n)\subset \Om$ such that $\tu(x_n)\to 0$ and
$$\lambda=\frac{g(\tu(x_n))\tu(x_n)-\frac{N}{4}h(\tu(x_n))\tu(x_n)}{|\tu(x_n)|^2}$$
for any $n\geq 1$. From (A5) there follows that
$$
\frac{g(\tu(x_n))\tu(x_n)}{|\tu(x_n)|^2}=\frac{H(\tu(x_n))}{|\tu(x_n)|^2}+\frac{2G(\tu(x_n))}{|\tu(x_n)|^2} \to 0
$$
as $n\to\infty$.
Hence
$$\lambda=-\lim_{n\to\infty}\frac{\frac{N}{4}h(\tu(x_n))\tu(x_n)}{|\tu(x_n)|^2}=-
\frac{N}{4}\lim_{n\to\infty}\Big(g'(\tu(x_n))-\frac{g(\tu(x_n))\tu(x_n)}{|\tu(x_n)|^2}\Big)=0$$
and we obtain a contradiction.\\
(b) Suppose that $g$ is odd. Then we may assume that $\tu$ is positive and radially symmetric. Then from Strauss lemma (\cite[Radial Lemma 1]{Strauss}) we may assume that $\tu$ is continuous and from \eqref{eq:mu} 
$$
\lambda \tu(x) = g(\tu (x))-\frac{N}{4} h(\tu(x))
$$
holds for $x \in \R^N$. Since $\tu$ is continuous and $\tu \in H^1 (\R^N)$, there is an interval $I$ such that $0 \in \overline{I}$ and
$$
\lambda u = g(u) - \frac{N}{4} h(u)\quad\hbox{ for } u \in \overline{I}.
$$
From the definition of $h$ we obtain that
$$
\lambda u = \left(1+\frac{N}{4} \right) g(u) - \frac{N}{4} g'(u)u\quad\hbox{ for } u \in \overline{I}.
$$
Hence
$$
g(u) = C_1 |u|^{\frac{4}{N}} u + C_2 u, \quad u \in \overline{I}
$$
for some $C_1, C_2 \in \R$. In particular $G(u) = \frac{C_1}{2+\frac{4}{N}} |u|^{2+\frac{4}{N}} + \frac{C_2}{2} u^2$. From (A1) there follows that $C_2 = 0$, $C_1 \geq 0$ and we obtain a contradiction with the first inequality in (A5,$\preceq$).

	Therefore $\mu\neq -1$ and  taking into account  Nehari and Pohozaev identities for \eqref{eq:mu},
	we obtain
\begin{eqnarray*}
\frac12(1+\mu)\int_{\R^N}|\nabla \tu|^2+\lambda|\tu|^2\, dx&=&\int_{\R^N} g(\tu)\tu+\frac{N}{4}\mu h(\tu)\tu\,dx,\\
(1+\mu)\int_{\R^N}|\nabla \tu|^2+\frac{2^*}{2}\lambda|\tu|^2\, dx&=&2^*\int_{\R^N} G(\tu)+\frac{N}{4}\mu H(\tu)\,dx,
\end{eqnarray*}	
thus
\begin{equation}\label{eqNewNP}
(1+\mu)\int_{\R^N}|\nabla \tu|^2\, dx=\frac{N}{2}\int_{\R^N} H(\tu)+\frac{N}{4}\mu\big(h(\tu)\tu-2H(\tu)\big)\,dx.
\end{equation}
	Since $\tu\in\cM$ we get 
	$$(1+\mu)\frac{N}{2}\int_{\R^N}H(\tu)\, dx=\frac{N}{2}\int_{\R^N} H(\tu)+\frac{N}{4}\mu\big(h(\tu)\tu-2H(\tu)\big)\,dx$$
	and
	$$\mu\int_{\R^N}h(\tu)\tu-\Big(2+\frac{4}{N}\Big)H(\tu)\,dx=0.$$
	In view of (A4) and since $\frac4NG(u)\preceq H(u)$ for $u \in\R$, then similarly as in proof of Lemma \ref{lem:SDineq} we obtain
	 $\mu=0$. Therefore $\tu$ solves \eqref{eq}. In the case (b) we already know that $\tu$ is nonnegative and radially symmetric. Hence, from the maximum principle, $\tu$ is positive and the proof is completed. In the case (a) note that our solution $\tu$ is a minimizer of $J$ subject to the following constraints
	 \begin{align}
	 \label{constr1}\int_{\R^N} | \tu |^2 \, dx &= \rho > 0, \\
	 \label{constr2}\int_{\R^N} |\nabla \tu|^2 + |\tu|^2 -\frac{N}{2} H(\tu) \, dx &= \rho > 0.
	 \end{align}
	 From the regularity theory we know that every minimizer of \eqref{eq} with respect to \eqref{constr1} and \eqref{constr2} is of class $\cC^1$ (see \cite[Appendix B]{Struwe}). Hence, from \cite[Theorem 2]{Maris} there follows that $\tu$ is radially symmetric with respect to a one-dimensional affine subspace $V$ in $\R^N$.
\end{altproof}

With the aid of Lemmas \ref{lem:c_0attained} and \ref{lem:SDineq} we easy infer that the ground state energy map \eqref{gsm} is strictly decreasing. The further properties are given as follow.

\begin{Prop}\label{prop:con} Under the assumptions of Theorem \ref{th:main},
the ground state energy map $\rho\mapsto \inf_{\cS\cap\cM}J$ is continuous, strictly decreasing and $\inf_{\cS\cap\cM}J\to \infty$ as $\rho\to 0^+$. If $\eta=0$ and
\begin{equation}\label{cr}
\lim_{u\to 0} G(u)/|u|^{2^*} =\infty,
\end{equation}
and $\rho\to\infty$, then $\inf_{\cS\cap\cM}J\to 0^+$.
\end{Prop}
\begin{proof}
Let us denote 
$$\cD_\rho := \left\{ u \in H^1 (\R^N) \ : \ \int_{\R^N} |u|^2 \, dx \leq \rho \right\}\hbox{ and }\cS_\rho := \left\{ u \in H^1 (\R^N) \ : \ \int_{\R^N} |u|^2 \, dx = \rho \right\}.$$
Suppose that $\rho_n\to \rho^+$ as $n\to\infty$, and let $J(u_n)=\inf_{\cD_{\rho_n}\cap \cM} J$ for some $u_n\in \cD_{\rho_n} \cap \cM$. Arguing as in proof of Lemma \ref{lem:c_0attained}, $u_n\weakto \tu$ such that $r(\tu)\geq 1$ up to a translation and up to a subsequence. If $r(\tu)>1$, then
\begin{eqnarray*}
	\inf_{\cD_\rho\cap\cM}J
	&\leq& J(\tu(r(\tu)\cdot))
	=r(\tu)^{-N}\frac{N}{4}\int_{\R^N}H(\tu)-\frac{4}{N}G(\tu)\, dx\\
	&<& \frac{N}{4}\int_{\R^N}H(\tu)-\frac{4}{N}G(\tu)\, dx\\
	&\leq& \liminf_{n\to\infty} \frac{N}{4}\int_{\R^N}H(u_n)-\frac{4}{N}G(u_n)\, dx\\
	&=& \liminf_{n\to\infty} J(u_n)\leq \inf_{\cD_\rho\cap\cM}J,
\end{eqnarray*}
where the last inequality holds, since $\cD_{\rho}\cap\cM\subset \cD_{\rho_n}\cap\cM$. We get a contradiction and $r(\tu)=1$ and as in proof of Lemma \ref{lem:c_0attained} we infer that $J(\tu)=\inf_{\cD_\rho\cap \cM} J =\lim_{n\to\infty} \inf_{\cD_{\rho_n}\cap \cM} J$. Suppose that $\rho_n\to \rho^-$ as $n \to \infty$, and choose $u\in \cD_\rho\cap\cM$ so that $J(u)=\inf_{\cD_\rho\cap\cM}J$. Similarly as in \cite[Lemma 3.1]{JeanjeanLuNorm} we consider $s_n:=\sqrt{\rho_n/\rho}$, $v_n:=s_nu$
 and in view of Lemma \ref{lem:phi} we find $\lambda_n$ such that $\lambda_n^{N/2}v_n(\lambda_n\cdot)\in\cM$, however $\lambda_n$ need not be unique and $(\lambda_n)$ may be divergent. Note that $|\lambda_n^{N/2}v_n(\lambda_n\cdot)|_2=|v_n|_2=\rho_n$. If $\lambda_n\to\infty$ passing to a subsequence, then by (A2)
 $$s_n^2\int_{\R^N}|\nabla u|^2\,dx=\frac{N}{2}\int_{\R^N}\frac{H(\lambda_n^{N/2}s_n u)}{\big(\lambda_n^{N/2}\big)^{2_*}}\, dx\to \infty,$$ 
 which is a contradiction with $s_n\to 1$, as $n\to\infty$. Similarly (A3) exclude 
 $\lambda_n\to 0$ passing to a subsequence. Therefore, passing to subsequence $\lambda_n\to \lambda >0$, $\lambda^{N/2} u(\lambda \cdot) \in \cM$ and 
 $$\lim_{n\to\infty}J(\lambda_n^{N/2}v_n(\lambda_n\cdot))=
 \frac{\lambda^2}{2}\int_{\R^N}|\nabla u|^2\,dx-\int_{\R^N}G(\lambda^{N/2} u)\lambda ^{-N}\, dx=J(\lambda^{N/2}u(\lambda\cdot ))=J(u),$$
 where the last equality follows from Lemma \ref{lem:phi}, hence
 $\limsup_{n\to\infty}\inf_{\cD_{\rho_n}\cap\cM}J\leq \inf_{\cD_{\rho}\cap\cM} J$ and taking into account that $\cD_{\rho_n}\cap\cM\subset \cD_\rho\cap\cM$ we conclude the continuity of the ground state energy map.

Suppose that $\rho_n \to 0^+$ and let $J(u_n)=\inf_{\cD_{\rho_n} \cap \cM} J$ for some $u_n\in \cS_{\rho_n}$. We follow the ideas from \cite[Lemma 3.5]{JeanjeanLuNorm}. Put $\lambda_n := \frac{1}{|\nabla u_n|_2} > 0$ and $v_n := \lambda_n^{N/2} u_n(\lambda_n \cdot)$. Then $|\nabla v_n|_2 = 1$, $|v_n|_2 = |u_n|_2 = \rho_n \to 0^+$, $u_n = \lambda_n^{-N/2} v_n(\lambda_n^{-1} \cdot) \in \cM$ and $(v_n)$ is bounded in $H^1 (\R^N)$. In particular, $(v_n)$ is bounded in $L^{2^*} (\R^N)$ and from the interpolation inequality there holds
$$
|v_n|_{2_*} \leq | v_n|_2^{\frac{2}{N+2}} |v_n|_{2^*}^{\frac{N}{N+2}} = \rho_n^{\frac{2}{N+2}} |v_n|_{2^*}^{\frac{N}{N+2}} \to 0\hbox{ as } n \to \infty.
$$
Hence $|v_n|_{2_*} \to 0$ and $\int_{\R^N}G(\lambda^{N/2} v_n)\lambda ^{-N}\, dx \to 0$ as $n \to \infty$ for any fixed $\lambda > 0$. From Lemma \ref{lem:phi} there follows that 
$$
J(u_n) = J \left( \lambda_n^{-N/2} v_n(\lambda_n^{-1} \cdot) \right) \geq J \left( \lambda^{N/2} v_n(\lambda \cdot) \right) = \frac{\lambda^2}{2}-\int_{\R^N}G(\lambda v_n)\lambda ^{-N}\, dx = \frac{\lambda^2}{2} + o(1)
$$
for any $\lambda > 0$. Hence $J(u_n) \to \infty$.

Suppose now that $\eta = 0$. Then the ground state energy map \eqref{gsm} is well-defined for all $\rho > 0$. Suppose that $\rho_n \to \infty$. Take $u \in H^1(\R^N)$ as a ground state solution for the problem with $\rho = 1$, i.e. $J(u) = \inf_{\cD_1 \cap \cM} J = \inf_{\cS_1 \cap \cM} J$. From the regularity theory we know that $u$ is continuous, and therefore $u \in L^\infty (\R^N)$. Without loss of generality we may assume that $\rho_n > 1$ and, as in \cite[Lemma 3.6]{JeanjeanLuNorm}, define $u_n := \sqrt{\rho_n} u$. Then $u_n \in \cS_{\rho_n} \subset \cD_{\rho_n}$. From Lemma \ref{lem:phi} there is $\lambda_n >0$ such that $v_n :=\lambda_n^{N/2} u_n (\lambda_n \cdot) \in \cM$. In general, $\lambda_n$ is not unique. Moreover $|u_n|_2 = |v_n|_2$ so that $v_n \in \cD_{\rho_n} \cap \cM$. Hence
$$
0 < \inf_{\cD_{\rho_n} \cap \cM} J \leq J(v_n) \leq \frac12 \int_{\R^N} |\nabla v_n|^2 \, dx = \frac12 \lambda_n^2 \rho_n  \int_{\R^N} |\nabla u|^2 \, dx,
$$
so it is enough to show that $\lambda_n \sqrt{\rho_n} \to 0$. Note that
\begin{align*}
\lambda_n^2 \rho_n \int_{\R^N} |\nabla u|^2 \, dx = \int_{\R^N} |\nabla v_n|^2 \, dx = \frac{N}{2} \int_{\R^N} H(v_n) \, dx = \frac{N}{2} \lambda_n^{-N} \int_{\R^N} H( \lambda_n^{N/2} \sqrt{\rho_n} u) \, dx
\end{align*}
and
\begin{align*}
\int_{\R^N} |\nabla u|^2 \, dx = \frac{N}{2} \lambda_n^{-N-2} \rho_n^{-1} \int_{\R^N} H( \lambda_n^{N/2} \sqrt{\rho_n} u) \, dx = \frac{N}{2} \rho_n^{2/N} \int_{\R^N} \frac{H( \lambda_n^{N/2} \sqrt{\rho_n} u)}{ \left|  \lambda_n^{N/2} \sqrt{\rho_n} u \right|^{2+\frac{4}{N}} } |u|^{2+\frac{4}{N}} \, dx.
\end{align*}
$$
\int_{\R^N} \frac{H( \lambda_n^{N/2} \sqrt{\rho_n} u)}{ \left|  \lambda_n^{N/2} \sqrt{\rho_n} u \right|^{2+\frac{4}{N}} } |u|^{2+\frac{4}{N}} \, dx \to 0\hbox{ as } n \to \infty,\hbox{ and }\lambda_n^{N/2} \sqrt{\rho_n} \to 0.
$$
Fix $\varepsilon > 0$. Then, from (A5) and \eqref{cr} there follows that
$$
H(s) \geq \frac{4}{N} G(s) \geq \eps^{-1} |s|^{2^*}
$$
for sufficiently small $|s|$. Then, taking into account that $u \in L^\infty (\R^N)$, for sufficiently large $n$
\begin{align*}
\int_{\R^N} |\nabla u|^2 \, dx &= \frac{N}{2} \lambda_n^{-N-2} \frac{1}{\rho_n} \int_{\R^N} H( \lambda_n^{N/2} \sqrt{\rho_n} u) \, dx \geq \eps^{-1} \frac{N}{2} \lambda_n^{-N-2} \frac{1}{\rho_n} \left| \lambda_n^{N/2} \sqrt{\rho_n} \right|^{2^*} |u|_{2^*}^{2^*} 
\\ &= \eps^{-1} \frac{N}{2} \lambda_n^{\frac{4}{N-2}} \rho_n^{\frac{2}{N-2}}|u|_{2^*}^{2^*} = \eps^{-1} \frac{N}{2} (\lambda_n^2 \rho_n)^{\frac{2}{N-2}}|u|_{2^*}^{2^*} 
\end{align*} 
and $\lambda_n^2 \rho_n \to 0$ as $n \to \infty$, which completes the proof.
\end{proof}

\section*{Acknowledgements}
\noindent The authors would like to thank Louis Jeanjean and Sheng-Sen Lu for valuable comments helping to improve the first version of this paper. The authors also thank the Department of Mathematics,
Karlsruhe Institute of Technology (KIT) in Germany, where a part of this work has been prepared, for the warm hospitality. Bartosz Bieganowski was partially supported by the National Science Centre, Poland (Grant No. 2017/25/N/ST1/00531). Jaros\l aw Mederski was partially supported by the National Science Centre, Poland (Grant No. 2017/26/E/ST1/00817).

\end{document}